\documentclass[12pt,a4paper,draft,twoside,reqno]{amsart}
\usepackage{amssymb}
\usepackage{amsmath}
\usepackage{amsfonts}
\usepackage{euscript}

\newtheorem{theorem}{Theorem}

\newtheorem{example}[theorem]{Example}

\newtheorem{proposition}[theorem]{Proposition}
\newtheorem{remark}[theorem]{Remark}

\begin{document}

\title[On invariants and equivalence %of differential operators under Lie pseudogroups actions
]{On invariants and equivalence of differential operators under Lie
pseudogroups actions}
\author{Valentin Lychagin \&  Valeriy Yumaguzhin}
\maketitle

\begin{abstract}
In this paper, we study invariants of linear differential operators with
respect to algebraic Lie pseudogroups. Then we use these invariants and the
principle of n-invariants to get normal forms (or models) of the
differential operators and solve the equivalence problem for actions of
algebraic Lie pseudogroups. As a running example of application of the
methods, we use the pseudogroup of local symplectomorphisms.
\end{abstract}

%%%%%%%%%%%%%%%%%%%%%%%%%%%%%%%%%%%%%%%%%
\section{Introduction}

The counterpoint and polyphony play the crucial role, not only in music and
art, but also in science and mathematics.

So, for example, if we take the theory of invariants, then at the very
beginning to study invariants of binary forms Sylvester J.J. proposed the
construction of invariants by transvectants (\cite{Syl},\cite{Olv}) and
Gordan P.(\cite{Gor}) not only proved that the algebra of polynomial
invariants of binary forms is finitely generated, but also shown that
invariants might be found by transvectants, that are concrete differential
operators.

It explains Gordon's denial of Hilbert's proof of the finiteness of the
invariant algebras.

In the previous publications (\cite{BL2},\cite{BL},\cite{LR}), we discussed
applications of differential invariants to finding of the algebraic ones. It
is related to differences between Hilbert-Rosenlicht (\cite{Ros}) and
Lie-Tresse theorems (\cite{KL}), describing the structures of algebras of
algebraic and differential invariants. This interplay between the algebraic
and differential methods seems extremely fruitful, especially in application
to differential operators.

It is also worth to note that the number of basic invariants, needed to
describe regular orbits, essentially different for algebraic and
differential invariants. Thus, the Rosenlicht theorem requires the number of
algebraic invariants equals to the codimension of the regular orbit, that
exponentially grows with the degree of algebraic forms under consideration.
On the other hand, the Lie-Tresse theorem, or the n-invariant principle (%
\cite{ALV}), requires essentially fewer invariants. Thus, in description of
linear scalar differential operators, we need only n, equals to the
dimension of the base, invariants.

In this paper, we study orbits of scalar linear (and some class of
non-linear) differential operators under action of a Lie pseudogroups, the
case of the complete pseudogroup of all local diffeomorphisms was studied in
(\cite{LY}).

We reverse, in this case, the interplay between algebraic and differential
invariants and show that to find the necessary number of differential
invariants we may use algebraic invariants of symbols of operators. Because
the number of these invariants restricted to the dimension of the base, we
use here transvectants.

The paper is organized as follows. At first, we remind the basics on Lie
pseudogroups of transformations and their invariants. Then we discuss the
principle of n-invariants and models of linear differential operators built
on invariants, being in general position. This allows us to get a solution
of the equivalence problem for differential operators with respect to a
given Lie pseudogroup of transformations. We illustrate this approach on the
Lie pseudogroup of symplectomorphisms. To this end, we introduce the
symplectic transfectants. They allow us to construct symplectic invariants
of differential operators, by using the symbols of operators only.

However, there is a class of operators where symplectic invariants of
symbols are never in general position. These operators are characterized by
a requirement that symplectic orbits of their symbols are regular and
constant. We call such operators as constant type operators (\cite{LY1}).

With such types of operators, having order $k > 2$ %k\TEXTsymbol{>}2
, we associate a
unique affine connection that is symplectic, i.e. preserves the symplectic
structure form, as well as the symbol. These connections we call Wagner
connections (\cite{Wag}), who discovered such connections associated with
cubic symmetric tensors on 2-dimensional manifolds. These connections have
zero curvature, but non-trivial torsion. We use these connections to split
(or quantize) the differential operators into the sum of symmetric tensors,
and get additional invariants by pure algebraic methods. For differential
operators of the second order, we use the Levi-Civita connection to split
the differential operator and find additional invariants.

%%%%%%%%%%%%%%%%%%%%%%%%%%%%%%%%%%%%%%%%%
\section{Lie pseudogroups and Lie equations}

By a pseudogroup $\EuScript{P}$, %$\mathcal{P}$, %$\EuScript{P}$, $\mathcal{P}$, $\Script{P,,}$, 
acting on a connected manifold $M$, we mean (\cite{Sing}) a collection of local diffeomorphisms, i.e. diffeomorphisms of open sets in %\thinspace 
$M$, closed under:

\begin{enumerate}
\item restrictions on open domains: $\phi\in\EuScript{P}$, $U\subset M$- open, then 
$\phi \vert _{U}\in\EuScript{P}$; 
%$\left. \phi \right\vert _{U}\in\EuScript{P}$ and 
if $\phi\vert _{U}\in\EuScript{P}$ for any
open domain $U\subset {\rm dom} (\phi)$, % $U\subset {\rm dom}$ 
 then $\phi \in \EuScript{P}$, %\mathfrak{\Gamma }$ .
%\end{enumerate}\end{document} 
\item composition: $\phi ,\psi \in \EuScript{P}$, then $\phi \circ
\psi \in \EuScript{P}$, if ${\rm dom}\left( \phi \right) \subset 
{\rm im}\left( \psi \right)$, %${\rm dom}\left( \phi \right) \subset 
%\limfunc{im}\left( \psi \right) .$
\item inverse: $\phi\in\EuScript{P}\Rightarrow \phi^{-1}\in \EuScript{P}$, and ${\rm id}_{M}\in \EuScript{P}$.
%$\limfunc{id}_{M}\in \EuScript{P}.$
\end{enumerate}

The action is said to be \ \textit{transitive,} if for any $a,b\in M$ ,\
there exists $\phi \in \EuScript{P},$ such that $\phi \left( a\right)
=b,$ and the action is said to be \textit{primitive,} if there are no $%
\EuScript{P}$ - invariant foliations on $M.$

Let $\delta _{k}\colon D_{k}\left( M\right) \rightarrow M$ be the bundles of $k$%
-jets of local diffeomorphisms. Fibres $\delta _{k}^{-1}\left( a\right) ,$
where $a\in M,$ consist of $k$-jets $[\phi ]_{a}^{k}$ of local
diffeomorphisms $\phi $ at the point.

We have the tower of bundles 
\begin{equation*}
\rightarrow\! D_{k}(M) \overset{\delta _{k,k-1}}{\rightarrow}%
D_{k-1}(M)\rightarrow \cdots\rightarrow\! D_{2}(M) 
\overset{\delta _{2,1}}{\rightarrow}D_{1}(M)\overset{\delta_{1,0}}{\rightarrow }M\times M\overset{\delta}{\rightarrow}M,
\end{equation*}%
where mappings $D_{k}\left( M\right) \overset{\delta _{k,k-1}}{\rightarrow }%
D_{k-1}\left( M\right) $ are the reductions of $k$-jets $[\phi ]_{a}^{k}$ to
the $\left( k-1\right) $-jets $[\phi ]_{a}^{k-1}.$

Given a pseudogroup $\EuScript{P},$ we define $G_{k}\subset
D_{k}\left( M\right) $ to be the following  subset, %consisting of $k$-jets of elements of $\mathcal{\Gamma }$ :%
\begin{equation*}
G_{k}=\left\{ [\phi ]_{a}^{k},\phi \in \EuScript{P},a\in {\rm dom}%
\left( \phi \right) \right\} .
\end{equation*}%
We say that the pseudogroup $\EuScript{P}$ is \textit{regular }if all $%
G_{k}$ are smooth submanifolds in $D_{k}\left( M\right) $ and $\delta
_{k}:G_{k}\rightarrow M$ are smooth subbundles of bundles $\delta
_{k}:D_{k}\left( M\right) \rightarrow M.$

Remark, that the composition in the pseudogroup induces partially defined
composition and defines a Lie algebroid structure in $G_{k}.$ Taking fibres $%
G_{k}$ at points $a\times a$ of the diagonal $\Delta \subset M\times M,$ we
get subgroups $G_{k}\left( a\right) $ of the differential groups $%
D_{k}\left( a\right) ,$ formed by $k$-jets $[\phi ]_{a}^{k}$ of
diffeomorphisms $\phi \in \EuScript{P},$ such that $\phi \left(
a\right) =a.$

We say that a regular pseudogroup $\EuScript{P}$ %$\mathfrak{G}$ 
is a \textit{Lie pseudogroup%
} if the tower 
\begin{equation}
\rightarrow G_{k}\overset{\delta _{k,k-1}}{\rightarrow }G_{k-1}\overset{%
\delta _{k-1,k-2}}{\rightarrow }G_{k-2}\rightarrow \cdots \rightarrow G_{1}%
\overset{\delta _{1,0}}{\rightarrow }M\times M\overset{\delta }{\rightarrow }%
M,  \label{Tower}
\end{equation}%
is the tower of smooth bundles, and it defines formally integrable equation $%
G$ (called \textit{Lie equation, }see\textit{\ \cite{Kum} }for more details)
in the sense that \ the first prolongations $G_{k}^{\left( 1\right) }\subset
G_{k+1}$, for all $k\geq 1.$

It also follows from the Cartan-Kuranishi prolongation theorem that in this
case there is a number $k_{0},$ called the \textit{order of the pseudogroup} 
$\EuScript{P},$ such that $i$-th prolongations $G_{k_{0}}^{\left(
i\right) }\ \ $of $G_{k_{0}}$ coincide with $G_{k_{0}+i},$ for all $i\geq 0.$

In other words, elements of $\EuScript{P}$ are solutions of
differential equation $G_{k_{0}},$ having order $k_{0}.$

Moreover, the Cartan--K\"{a}hler theorem states that in the analytical case,
i.e. in the case when all manifolds and mappings under consideration are
analytic, the elements of pseudogroup $\EuScript{P}$ are precisely all
local solutions of equation $G_{k_{0}}.$

This restriction on $\EuScript{P}$ we'll use only once in the paper
(and we'll indicate it) so in the rest of the paper the pseudogroup will be
a pseudogroup of smooth local diffeomorphisms.

It is also worth to note that fibres of the bundles $D_{k}\left( M\right) $
have the intrinsic structure of algebraic manifolds and the prolongations of
local diffeomorphisms of the manifold $M$ to these bundles are algebraic too.

In particular, the differential groups $D_{k}\left( a\right) $ are algebraic
as well as their actions$.$

We say, (see \cite{KL}, for more details), that the Lie pseudogroup $%
\EuScript{P}$ is \textit{algebraic} if $G_{k}\subset D_{k}\left(
M\right) $ inherit this algebraicity, i.e. fibres $G_{k}$ are algebraic
submanifolds in the fibres $D_{k}\left( M\right) $ and the composition law
is also algebraic.

In particular, groups $G_{k}\left( a\right) $ are algebraic as well as their
actions$.$

Remark also, that this notion of algebraicity could also be applied to jet
bundles $\pi _{k}:J^{k}\left( \pi \right) \rightarrow M$ of sections of any
smooth bundle $\pi :E\left( \pi \right) \rightarrow M.$ Namely,
prolongations of point transformations (i.e. local diffeomorphisms of $%
E\left( \pi \right) $ ) into fibres of projections $\pi _{k,0}:J^{k}\left(
\pi \right) \rightarrow E\left( \pi \right) $ are rational in the standard
jet coordinates, and therefore the algebraic structure of these fibres
induced by the choice of the standard jet coordinates.

We also will call formally integrable differential equations $\mathcal{%
E\subset }J^{k}\left( \pi \right) $ \textit{algebraic }if the bundle $\pi
:E\left( \pi \right) \rightarrow M$ is natural and algebraic, i.e. fibres of 
$\pi $ have the structure of irreducible algebraic manifolds that is
invariant under the action of the pseudogroup of local diffeomorphisms of $M$
and fibres of the projection $\pi _{k,0}:\mathcal{E\rightarrow }E\left( \pi
\right) $ are irreducible algebraic manifolds too.

In such a setting, we define (\cite{KL})\textit{\ rational differential }$%
\EuScript{P}-$\textit{invariants of order} $\left( k+l\right) $ as $%
\EuScript{P}-$invariant rational functions on the $l$-th prolongation $%
\mathcal{E}^{\left( l\right) }\mathcal{\subset }J^{k+l}\left( \pi \right) .$

Because we assumed that the $\EuScript{P}$ - action is transitive on
manifold $M$, these invariants are completely defined by their values on
the fibre $\mathcal{E}_{b}^{\left( l\right) }\mathcal{\subset }%
J_{b}^{k+l}\left( \pi \right) $ over\ a fixed base point $b\in M.$

Thus, rational differential $\EuScript{P}-$invariants of order $\left(
k+l\right) $ form a field $\mathcal{F}_{l}$ that, due to the Rosenlicht
theorem (\cite{Ros}), separates regular $G_{k+l}\left( b\right)$ - orbits in 
$\mathcal{E}_{b}^{\left( l\right) }$ and the transcendence degree of $%
\mathcal{F}_{l}$ equals to the codimension of the regular $G_{k+l}\left(
b\right) -$ orbits.

Moreover, the Lie -Tresse theorem (\cite{KL}) states that the field of all
rational differential $\EuScript{P}-$invariants $\mathcal{F}_{\ast }$
is generated by a finite number of $\EuScript{P}-$invariants and $%
\EuScript{P}-$invariant derivations (usually, they are Tresse
derivations, associated with a set of rational differential %$\mathfrak{\Gamma }$
$\EuScript{P}$-invariants, being in general position).

%%%%%%%%%%%%%%%%%%%%%%%%%%%%%%%%%%%%%%%%%%
\section{Linear differentiial operators and their $\EuScript{P}$%
-invariants}

In the paper we use the following notations, the $C^{\infty }\left( M\right) 
$-modules of smooth sections of vector bundles $\pi :E\left( \pi \right)
\rightarrow M$ will be denoted by $C^{\infty }\left( \pi \right) .$

By $\pi _{k}:J^{k}\left( \pi \right) \rightarrow M$ \ we denote the bundles
of $k$-jets of sections of the bundle $\pi $, and by $\pi _{k,l}:J^{k}\left(
\pi \right) \rightarrow J^{l}\left( \pi \right) ,k>l,$ the natural
projections.

Let $\psi _{k}\colon\mathit{Diff}_{k}(M) \rightarrow M$ be the bundle of
linear scalar differential operators of order $\leq k.$

Then $C^{\infty }\left( \psi _{k}\right) ={\rm Diff}_{k}\left( M\right) $
is the module of linear scalar differential operators on $M,$ having order $%
\leq k.$

The following exact sequences of the modules will play the crucial role 
\begin{equation}
0\rightarrow {\rm Diff}%\nolimits
_{k-1}\left( M\right) \rightarrow 
{\rm Diff}%\nolimits
_{k}\left( M\right) \overset{{\rm smbl}}{%{\limfunc{smbl}}{%
\rightarrow }\Sigma _{k}\left( M\right) \rightarrow 0.  \label{exact seq}
\end{equation}%
Here $\Sigma _{k}\left( M\right) $ is the module of symmetric vector fields
of degree $k$, $\Sigma _{k}(M) = C^{\infty }(\tau_{k})$, where $\tau _{k}:S^{k}\left( TM\right) \rightarrow M$ is the $%
k$-th symmetric power of the tangent bundle and the ${\rm smbl}$-map %\limfunc{smbl}$-map
sends operators $A\in {\rm Diff}_{k}\left( M\right) $ to their symbols $%
\sigma _{A}\in \Sigma _{k}\left( M\right) $.

Remark, that any Lie pseudogroup $\EuScript{P}$ of local
diffeomorphisms of $M$ acts also on all these bundles.

Rational functions on manifolds $J^{l}\left( \psi _{k}\right) $ or $%
J^{l}\left( \tau _{k}\right) ,$ that are invariant with respect to the
prolonged $\EuScript{P}$-actions, will be called \textit{rational }$%
\EuScript{P}$\textit{-invariants }of differential operators or $k$%
-symmetric vector fields.

Because of transitivity $\EuScript{P}$-action on $M$ these invariants
are defined by their values on the fibres $J_{b}^{l}\left( \psi _{k}\right) $
or $J_{b}^{l}\left( \tau _{k}\right) $ at a fixed basic point $b\in M.$

We denote by $\mathcal{F}_{k,l},$ or simpler $\mathcal{F}_{l}$ ,if the order
of operators under consideration is fixed, the field of rational $\EuScript{%
P}$ \textit{-invariants of order }$\leq l,$ and by $\mathcal{F}%
_{k,l}^{\sigma }\ $ we denote the field of rational $\EuScript{P}$%
\textit{-invariants of order }$\leq l$ of the symmetric vector fields.

Exact sequence (\ref{exact seq}) shows that $\mathcal{F}_{k,l}^{\sigma }$
are subfields of $\mathcal{F}_{k,l}$ and this gives us an option to get $%
\EuScript{P}$-invariants of differential operators in pure algebraic
way.

In addition, we have the natural universal differential operator (\cite{LY}) 
\begin{eqnarray*}
\square &:&C^{\infty }\left( J^{\infty }\left( \psi _{k}\right) \right)
\rightarrow C^{\infty }\left( J^{\infty }\left( \psi _{k}\right) \right) , \\
\square &:&C^{\infty }\left( J^{l}\left( \psi _{k}\right) \right)
\rightarrow C^{\infty }\left( J^{l+k}\left( \psi _{k}\right) \right) ,l\geq
0,
\end{eqnarray*}%
that allow us the extend the set of available $\EuScript{P}$%
-invariants.

This operator is natural (and thus it is $\EuScript{P}$-invariant),
i.e. commutes with prolongations of diffeomorphisms of $M,$ and, therefore,
defines maps%
\begin{equation*}
\square :\mathcal{F}_{k,l}\rightarrow \mathcal{F}_{k,k+l},
\end{equation*}%
for $l\geq 0.$

Remark (see \cite{LY}), that in standard jet coordinates $\left(
x_{1},..,x_{n},u_{\alpha }\right)$, where $\left( x_{1},..,x_{n}\right) $
are local coordinates on $M,$ $\alpha $ are multi indices of lengths $%
\left\vert \alpha \right\vert \leq k,$ and $u_{\alpha }\left( A\right)
=A_{\alpha },$ if $A=%\dsum\limits_{\alpha }
\sum\limits_{\alpha }A_{\alpha }\partial ^{\alpha }$
in the coordinates $\left( x_{1},..,x_{n}\right) ,$ we have 
\begin{equation*}
\square =%\dsum
\sum\limits_{\left\vert \alpha \right\vert \leq k}
u_{\alpha }\frac{d^{\left\vert \alpha \right\vert }}{dx^{\alpha }},
\end{equation*}%
where $\displaystyle{\frac{d}{dx_{i}}}$ are the total derivatives.

The \textit{n-invariants principle} that we have used in (\cite{LY}) for the
pseudogroup of all local diffeomorphisms of $M,$ could be applied
practically word by word for actions of arbitrary Lie pseudogroups.

Namely, let $I_{1},\ldots,I_{n}\in C^{\infty }\left( J^{l}\left( \psi
_{k}\right) \right)$, $n=\dim M$, be $\EuScript{P}$-invariants that
are in general position in an open set $\mathcal{O\subset }$ $J^{l}\left(
\psi _{k}\right) $, i.e.%
\begin{equation*}
\widehat{d}I_{1}\wedge \cdots \wedge \widehat{d}I_{n}\neq 0
\end{equation*}%
over $\mathcal{O}$, where $\widehat{d}$ is the total differential.

Then, 
\begin{equation*}
I_{\alpha }=\frac{1}{\alpha !}\square \left( I^{\alpha }\right) ,
\end{equation*}%
where $I^{\alpha }=I_{1}^{\alpha _{1}}\cdots I_{n}^{\alpha _{n}},$ $%
\left\vert \alpha \right\vert \leq k,$ are also $\EuScript{P}$%
-invariants.

Moreover, their values at the operator $A$ (i.e. their restrictions on the $%
l $-jets of section $S_{A}$ that corresponts to operator $A$) coincide with
coefficients $A_{\alpha }$ of the operator $A$ in local coordinates $%
x_{1}=I_{1}\left( A\right) ,...,x_{n}=I_{n}\left( A\right) .$

Therefore, we get the following \ description of the $\EuScript{P}$%
-invariants.

\begin{theorem}
Let $I_{1},\ldots,I_{n}\in C^{\infty }\left( J^{l}\left( \psi _{k}\right)
\right)$, $n=\dim M$, be rational $\EuScript{P}$-invariants of scalar
linear differential operators on $M$, that are in general position in an
open set $\mathcal{O\subset }$ $J^{l}\left( \psi _{k}\right)$. %\newline

Then all rational $\EuScript{P}$-invariants of scalar linear
differential operators over the open set $\mathcal{O}$ are just rational
functions of invariants $(I_1,\ldots, I_n$, $I_{\alpha}, |\alpha|\leq k )$ and their Tresse derivatives $\displaystyle{\frac{d^{\beta}I_{\alpha }}{dI^{\beta }}}$.
\end{theorem}

%%%%%%%%%%%%%%%%%%%%%%%%%%%%%%%%%%%%%%%%%%
\section{ $\EuScript{P}$-models and $\EuScript{P}$-equivalence
of linear differential operators}

The above theorem allow us to construct $\EuScript{P}$-models (or $%
\EuScript{P}$-normal forms) of linear differential operators.

Namely, similar to (\cite{LY}), we consider space $\Phi _{k}=\mathbb{R}%
^{n}\times \mathbb{R}^{\binom{n+k}{k}}$ with coordinates $\left(
y_{1},..,y_{n},Y_{\alpha },\left\vert \alpha \right\vert \leq k\right) .$

Let $I_{1},\ldots,I_{n}\in C^{\infty }\left( J^{l}\left( \psi _{k}\right)\right)$ 
be rational $\EuScript{P}$-invariants in general position for an
open set $\mathcal{O\subset }$ $J^{l}\left( \psi _{k}\right)$. Then any
linear differential operator $A\in {\rm Diff}_{k}\left( M\right) $ in a
domain $\mathcal{O}^{\prime }\subset M,$ where $S_{A}\left( \mathcal{O}%
^{\prime }\right) \subset $ $\mathcal{O}$, defines a map%
\begin{eqnarray*}
\phi _{B} &:&\mathcal{O}^{\prime }\rightarrow \Phi_{k}, \\
\phi _{B} &:&x\in \mathcal{O}^{\prime }\mathcal{\rightarrow }\left(
y_{1}=I_{1}\left( A\right) ,...,y_{n}=I_{n}\left( A\right) ,Y_{\alpha
}=I_{\alpha }\left( A\right) \right) .
\end{eqnarray*}%
We call a pair $\left( A,\mathcal{O}^{\prime }\right) $ \textit{adjusted} if
functions $\left( I_{1}\left( A\right) ,\ldots,I_{n}\left( A\right) \right)$
are coordinates in the domain $\mathcal{O}^{\prime }$, submanifold $\Sigma
_{A}=\phi _{A}\left( \mathcal{O}^{\prime }\right) \subset \Phi_{k}$ we call a 
$\EuScript{P}$-\textit{model of the operator} in the domain $\mathcal{O}^{\prime }$.

Let $B\in {\rm Diff}_{k}\left( M\right)$ be another differential
operator such that the pair $\left( B,\mathcal{O}^{\prime \prime }\right) $
is adjusted for the same $\EuScript{P}$-invariants $I_{1},\ldots, I_{n}$
and models operators $A$ and $B$ coincide, 
\begin{equation}
\phi _{A}\left( \mathcal{O}^{\prime }\right) =\phi _{B}\left( \mathcal{O}%
^{\prime \prime }\right) .  \label{mod}
\end{equation}%
Let $\psi _{AB}:\mathcal{O}^{\prime }\rightarrow \mathcal{O}^{\prime \prime
} $ be a diffeomorphism, such that $\psi _{AB}^{\ast }\left( I_{i}\left(
B\right) \right) =I_{i}\left( A\right) .$

\begin{theorem}[n-invariants principle]
Let $\left( A,\mathcal{O}^{\prime }\right) $ and $\left( B,\mathcal{O}%
^{\prime\prime}\right) $ be adjusted pairs for the same set of $\EuScript{P}$-invariants $I_{1},\ldots, I_{n}$. 

Then operators $A,B\in {\rm Diff}_{k}\left( M\right) $ are $\EuScript{P}$-equivalent in the open sets $\mathcal{O}^{\prime }$ and $\mathcal{O}^{\prime\prime}$ if and only if their 
$\EuScript{P}$-\textit{models coincide (\ref{mod}) and }$\psi _{AB}\in\EuScript{P}$.
\end{theorem}

\begin{remark}
In the case, when the Lie equation completely defines the pseudogroup $\EuScript{P}$, the last requirement of the theorem could be made to be
more constructive. Namely, diffeomorphism $\psi _{AB}$ should satisfy the
Lie equation.
\end{remark}

%%%%%%%%%%%%%%%%%%%%%%%%%%%%%%%%%%%%%%%%%%%
\section{$\EuScript{P}$- invariants of symbols}

At first, we remark, that symbols of linear differential operators from $%
{\rm Diff}_{k}\left( M\right) ,$ at a point $a\in M,$ are symmetric
tensors $\sigma \in S^{k}T_{a}\left( M\right) $.

To describe $\EuScript{P}$- invariants of symbols, we'll fix a base
point $b\in M$ and denote by $V$ the cotangent space $T_{b}^{\ast }\left(
M\right) .$

Then the symbols of differential operators are elements of the symmetric
power $S^{k}V^{\ast },$ and, hence, could be considered as homogeneous
polynomial functions of degree $k$ on vector space $V.$

Thus, the action of the pseudogroup $\EuScript{P}$ on symbols is
reduced to the action of the algebraic Lie group $G_{1}\left( b\right)
\subset %\limfunc{GL}
{\rm GL}\left( V^{\ast }\right) $ on the space of these
polynomials.

There are various algorithmic methods to find such invariants, see, for
examples (\cite{Dec},\cite{Olv},\cite{Pop}).

We proposed in (\cite{BL},\cite{LR}) methods of finding algebraic invariants
by using differential invariants. Thus, this paper, in some sense, conclude
the interplay between differential and algebraic invariants.

The most important are actions of primitive pseudogroups, where the Lie
algebras $\mathfrak{g}_{1}$ ( also called \textit{linear isotropy algebras})
of the correspondent Lie groups $G_{1}\left( b\right) $ are given in the
Cartan's classification list (see,\cite{Sing},\cite{GQS}).

W'll consider here the most common cases of these Lie algebras: $\mathfrak{g%
}_{1}=\mathfrak{sl}\left( V^{\ast }\right)$ %\limfunc{sl}
$\mathfrak{sl}\left( V^{\ast }\right) $ for the pseudogroup of volume
preserving diffeomorphisms, and $\mathfrak{g}_{1}=%\limfunc{sp}
\mathfrak{sl}\left( V^{\ast}\right)$ for the pseudogroup $\mathfrak{S}$ %$\mathfrak{S}$ % , $\mathfrak{P}$, $\mathfrak{\Gamma}$,%, %$\EuScript{Sp}$, ${\EuScript Sp}$, ${\EuScript SP}$, $\EuScript{S}$$\EuScript{P}$, $\mathcal{Sp}$ 
of symplectomorphisms.

The cases $\mathfrak{g}_{1}=%\limfunc{sl}
\mathfrak{sl}\left( V^{\ast }\right) $ and $%
\mathfrak{g}_{1}=%\limfunc{gl}
\mathfrak{gl}\left( V^{\ast }\right) $ were considered in (%
\cite{LY}), for this reason, we consider here the cases $\mathfrak{g}_{1}=%\limfunc{sp}
\mathfrak{sp}\left( V^{\ast }\right) $ (and also $\mathfrak{g}_{1}=%\limfunc{so%
\mathfrak{so}\left( V^{\ast }\right) )$ to illustrate the method of transvectants.

Remark, that the codimensions of regular $\mathfrak{g}_{1}-$ orbits in $%
S^{k}T\left( M\right) $ significantly exceed $\dim M,$ when $k>2,$ and,
therefore, we may use invariants of the symbols in realization of the $n$%
-invariant principle.

This observation explains the use of transvectants because we are interested
in practical methods of finding $n$ invariants, but not in the problem of
finding all possible invariants.

The case of the second order operators could be elaborated by using of the
Levi-Civita connections and related to them quantizations, as it was done in
(\cite{LY}).

\subsection{Symplectic transvectants}

Let $\left( M,\omega \right) $ be a symplectic manifold, where $\omega \in
\Omega ^{2}\left( M\right) $ is the structure form, and $\EuScript{P}$ = %
$\mathfrak{S}$ the pseudogroup of local symplectomorphisms.

Then the Lie equation, defining this pseudogroup, is the following 
\begin{equation*}
%\mathfrak{S} = \left\{ \phi mathfrak{\in }%\limfunc{Diffeo}
\mathfrak{S} = \left\{ \phi \in {\rm Diffeo}(M), \phi _{\ast }(\omega) =\omega \right\} ,
\end{equation*}%
and the linear isotropy algebra $\mathfrak{g}_{1}=%\limfunc{sp}
{\rm sp}\left( V^{\ast
}\right) .$

Let $\left\{ e_{1},..,e_{n},f_{1},...,f_{n}\right\} \subset V=T_{b}^{\ast
},\dim M=2n,$ be a canonical basis for the structure form $\omega _{b}\in
\Lambda ^{2}\left( V\right) ,$ i.e.%
\begin{equation*}
\omega _{b}=%\dsum
\sum\limits_{i=1}^{n}e_{i}\wedge f_{i}.
\end{equation*}%
Then,as we have seen, symbols of the linear differential operators on $M$ at
the point $b\in M,$ are symmetric tensors $\sigma \in S^{k}V^{\ast },$ that
we consider as homogeneous polynomials on $V.$

Denote by $\mathcal{S}=\oplus_{k\geq 0}S^{k}V^{\ast }$ the polynomial
algebra on $V$, and the structure form $\omega _{b}$ we present as
bi-differential operator 
\begin{equation*}
\widehat{\omega }:S\mathcal{\otimes S\rightarrow S\otimes S},
\end{equation*}%
that acts as follows%
\begin{equation*}
\widehat{\omega }\left( P\otimes Q\right) =\frac{1}{2}%\dsum
\sum\limits_{i=1}^{n}%
\left( e_{i}\left( P\right) \otimes f_{i}\left( Q\right) -f_{i}\left(
P\right) \otimes e_{i}\left( Q\right) \right) ,
\end{equation*}%
where we denoted by $e_{i}\left( P\right) ,f_{i}\left( P\right) $ the
directional derivatives of polynomial $P$ along the basis vectors $e_{i}$
and $f_{i}.$

Then by \textit{symplectic transvectant} of order $r$ we mean the following
bi-differential operator%
\begin{equation*}
P\otimes Q\in S\otimes\mathcal{S}\rightarrow [ P, Q_r ]_r %lbrack }P,Q\mathcal{]}_{r}
\in\mathcal{S},
\end{equation*}%
where 
\begin{equation*}
[P, Q]_{r}=\mu \big( \widehat{\omega }^{r}(P\otimes Q) \big) ,
\end{equation*}%
and $\mu :\mathcal{S\otimes S\rightarrow S}$ is the multiplication in the
algebra $\mathcal{S}.$

Remark that operators $\widehat{\omega }$ and $\mu$, as well as the
symplectic transfectants, are $%\limfunc
{\rm sp}-$ invariants.

In the canonical coordinates $\left( x_{1},...,x_{n},y_{1},...,y_{n}\right) $
on $V$ the symplectic transfectants have the following form:%
\begin{multline*}%{equation*}
[P,Q]_{r}\!=\! 2^{-r}%\dsum
\sum\limits_{l=0}^{r}%\dsum
\sum\limits_{l_{1}+\ldots
+l_{n}=l}%\dsum
\sum\limits_{m_{1}+\ldots +n_{n}=r-l}\left( -1\right) ^{r-l}\binom{%
r}{l}\binom{l}{l_{1}...l_{n}}\\
\times\binom{r-l}{m_{1}\cdots m_{n}}
\frac{\partial^{r}P}{\partial x^{l}\partial y^{m}}\frac{\partial ^{r}Q}{\partial
x^{m}\partial y^{l}},
\end{multline*}%{equation*}%
where%
\begin{equation*}
\frac{\partial ^{r}P}{\partial x^{l}\partial y^{m}}=\frac{\partial ^{r}P}{%
\partial x_{1}^{l_{1}}\cdots \partial x_{n}^{l_{n}}\partial
y_{1}^{m_{1}}\cdots \partial y_{n}^{m_{n}}}.
\end{equation*}

Remark, that 
\begin{equation*}
\deg \left( [P,Q]_{r}\right) =\deg \left( P\right) +\deg \left( Q\right) -2r,
\end{equation*}%
$[P,Q]_{1}$ coincides with the Poisson bracket, and 
\begin{multline*}%{eqnarray*}
4[P,Q]_{2}= %&=&%\dsum
\sum\limits_{i=1}^{n}\left( \frac{\partial ^{2}P}{\partial
x_{i}^{2}}\frac{\partial ^{2}Q}{\partial y_{i}^{2}}+\frac{\partial ^{2}Q}{%
\partial x_{i}^{2}}\frac{\partial ^{2}P}{\partial y_{i}^{2}}\right)\\
+2%\dsum
\sum\limits_{i\neq j}\left( \frac{\partial ^{2}P}{\partial x_{i}\partial
x_{j}}\frac{\partial ^{2}Q}{\partial y_{i}\partial y_{j}}+\frac{\partial
^{2}Q}{\partial x_{i}\partial x_{j}}\frac{\partial ^{2}P}{\partial
y_{i}\partial y_{j}}\right)  %&&
-2%\dsum
\sum\limits_{i,j}\frac{\partial ^{2}P}{\partial x_{i}\partial y_{j}}%
\frac{\partial ^{2}Q}{\partial y_{i}\partial x_{j}}.
\end{multline*}%{eqnarray*}%
The following statement follows directly from the definition of the
transfectant.

\begin{proposition}
The transvectants mappings $P\times Q\rightarrow $ $[P,Q]_{k}$ are bilinear
symmetric mappings if the order $k$ is even, and skew symmetric if the order 
$k$ is odd.They are trivial if $k>\min \left( \deg P,\deg Q\right)$.
\end{proposition}

Thus symplectic invariants of symbols, having order $k,$ are %\limfunc
${\rm sp}%
\left( V\right)$ - invariant polynomial functions on the space of symbols $%
S^{k}V^{\ast }$.

Namely, they do produce zero order symplectic differential invariants of
linear differential operators.

To find these $%\limfunc
{\rm sp}$-invaiants we take a symbol $P\in S^{p}V^{\ast
}, $having order $p,$ and remark that the transfectants generate the linear
operators 
\begin{equation*}
Q\in S^{q}V^{\ast }\rightarrow \lbrack P,Q]_{k}\in S^{p+q-2k}V^{\ast },
\end{equation*}%
where $k\leq \min \left( p,q\right) .$

Assume now, that the order $p$ is even. Then, we get operators%
\begin{equation*}
A_{P,q}:S^{q}V^{\ast }\rightarrow S^{q}V^{\ast },\ \ A_{P,q}:Q\rightarrow
\lbrack P,Q]_{\frac{p}{2}},
\end{equation*}%
for all $q\geq \displaystyle\frac{p}{2}$.

For general order $p,$ we substitute tensor $P\in S^{p}V^{\ast }$ by the
transfectants $P_{2l}=[P,P]_{2l},$ where $2l<p,$ and get operators 
\begin{equation*}
A_{P,l}:S^{q}V^{\ast }\rightarrow S^{q}V^{\ast },\ \ A_{P,l}:Q\rightarrow
\lbrack P_{2l},Q]_{p-2l},
\end{equation*}%
where $q\geq p-2l.$

\begin{theorem}
Functions $I_{l,k}:S^{p}V^{\ast }\rightarrow \mathbb{R}$, 
\begin{equation*}
I_{l,k}\left( P\right) =%\limfunc
{\rm Tr}\left( A_{P,l}^{k}\right) ,
\end{equation*}%
where $p-q\leq 2l\leq p,\ 1\leq k\leq \binom{2n+q-1}{q},$ for general degree 
$p,$and 
\begin{equation*}
J_{k,q}\left( P\right) =%\limfunc
{\rm Tr}A_{P,q}^{k},
\end{equation*}%
where $2q\geq p,\ 1\leq k\leq \displaystyle\binom{2n+q-1}{q}$, for even degree $p,$ are $%\limfunc
\frak{sp}$ - invariant polynomials on $S^{p}V^{\ast }$ of degree 2k.%
\newline
\newline
Bilinear forms 
\begin{equation*}
P\otimes Q\rightarrow \lbrack P,Q]_{p}\in \mathbb{R}
\end{equation*}%
are nondegenerate $%\limfunc
{\rm sp}-$ invariant skew symmetric 2-forms on $%
S^{p}V^{\ast },$ if $p$ is odd, and symmetric if $p$- even.
\end{theorem}

\begin{remark}
This theorem shows that the maximal degree of these invariant polynomials
equals to 
\begin{equation*}
k_{\max }=\binom{3n-1}{n},
\end{equation*}%
and the minimal is $k_{\min }=2.$
\end{remark}

%======================================================
\subsection{Regular %\limfunc
${\rm sp}$-orbits}

The Lie algebra $%\limfunc
{\rm sp}\left( V\right) $ we identify with space $%
S^{2}\left( V^{\ast }\right) $ equipped with the Poisson bracket $P,Q\in
S^{2}\left( V^{\ast }\right) \rightarrow \left( P,Q\right) _{1}\in
S^{2}\left( V^{\ast }\right) $ given by the first order transvectant .

Let $\left( x_{1},...,x_{n},y_{1},...,y_{n}\right) $ be the canonical
coordinates in the symplectic space $V,$ then we choose the basic in $%
S^{2}\left( V^{\ast }\right) $ as union of the following disjoint sets $%
B^{+}\cup B^{-}\cup C,$ where $B^{+}=\left\{ y_{i}y_{j},i,j=1,...,n,i\leq
j\right\} ,B^{-}=\left\{ x_{i}x_{j},i,j=1,...,n,i\leq j\right\} ,C=\left\{
x_{i}y_{j},i,j=1,...,n\right\} .$

Any quadric $Q\in S^{2}\left( V^{\ast }\right) $ generates the Hamiltonian
derivation $X_{Q}:S^{k}\left( V^{\ast }\right) \rightarrow S^{k}\left(
V^{\ast }\right) ,$ where $X_{Q}\left( P\right) =[P,Q]_{1},$ or in canonical
coordinates 
\begin{equation*}
X_{Q}=%\dsum
\sum\limits_{i=1}^{n}\left( \frac{\partial Q}{\partial y_{i}}\frac{%
\partial }{\partial x_{i}}-\frac{\partial Q}{\partial x_{i}}\frac{\partial }{%
\partial y_{i}}\right) .
\end{equation*}

Then vector fields 
\begin{eqnarray*}
b_{ij}^{+} &=&X_{y_{i}y_{j}}=y_{i}\frac{\partial }{\partial x_{j}}+y_{j}%
\frac{\partial }{\partial x_{i}}, \\
b_{ij}^{-} &=&X_{x_{i}x_{j}}=-x_{i}\frac{\partial }{\partial y_{j}}-x_{j}%
\frac{\partial }{\partial y_{i}}, \\
c_{ij} &=&X_{x_{i}y_{j}}=x_{i}\frac{\partial }{\partial x_{j}}-y_{j}\frac{%
\partial }{\partial y_{i}},
\end{eqnarray*}%
give us a basis into Lie algebra $%\limfunc
{\rm sp}\left( V\right) .$

To estimate dimensions of $%\limfunc
{\rm sp}$-orbits in to $S^{k}\left( V^{\ast
}\right) ,$ we have to estimate dimensions of subspaces $\left\{ B^{+}\left(
P\right) ,B^{-}\left( P\right) ,C\left( P\right) \right\} $ into $%
S^{k}\left( V^{\ast }\right) ,$ generated a polynomial $P\in S^{k}\left(
V^{\ast }\right) .$

To this end, we denote by $S_{i}\left( x\right) $ and $S_{j}\left( y\right) $
the spaces of homogeneous polynomials in $x$ and $y,$ having degrees $i$ and 
$j$ respectively.

Then we have the following direct decomposition 
\begin{equation*}
S^{k}\left( V^{\ast }\right) =\oplus _{\alpha +\beta =k}S_{\alpha ,\beta },
\end{equation*}%
where%
\begin{equation*}
S_{\alpha ,\beta }=S_{\alpha }\left( x\right) \otimes S_{\beta }\left(
y\right) ,
\end{equation*}%
and 
\begin{eqnarray*}
b_{ij}^{+} &:&S_{\alpha \beta }\rightarrow S_{\alpha -1,\beta +1},\  \\
b_{ij}^{-} &:&S_{\alpha \beta }\rightarrow S_{\alpha +1,\beta -1},\ \  \\
c_{ij} &:&S_{\alpha \beta }\rightarrow S_{\alpha ,\beta }.
\end{eqnarray*}

Take now a polynomial $P\in S^{k}\left( V^{\ast }\right) $ of the form: $%
P=P_{0}+P_{1},$ where $P_{0}\in S_{k}\left( x\right) ,P_{1}\in S_{k}\left(
y\right) .$

Then,%
\begin{eqnarray*}
B^{+}\left( P\right) &=&B^{+}\left( P_{0}\right) \subset S_{k-1},_{1},\  \\
B^{-}\left( P\right) &=&B^{-}\left( P_{1}\right) \subset S_{1,k-1}, \\
C\left( P_{0}\right) &\subset &S_{k,0},\ C\left( P_{1}\right) \subset
S_{0,k}.
\end{eqnarray*}%
Thus, $B^{+}\left( P_{0}\right) $ and $B^{-}\left( P_{1}\right) ,$ belong to
different vector spaces and are linear independent, when $k\geq 3,$ and
therefore 
\begin{equation*}
\dim \left( B^{+}\left( P_{0}\right) \oplus B^{-}\left( P_{1}\right) \right)
=n\left( n+1\right) ,
\end{equation*}%
for general polynomials $P_{0},P_{1}.$

Moreover, dimensions $C\left( P_{0}\right) \subset S_{k,0}$ and $C\left(
P_{1}\right) \subset S_{0,k}$ equals of dimensions of $%\limfunc
\frak{gl}\left(
n\right) $-orbits of $P_{0}\in S_{k}\left( x\right) \ $and $P_{1}\in
S_{k}\left( y\right) $ and, therefore, equals $n^{2},$ when $k\geq 3.$

To see this, it is enough to take polynomials $P_{0}\left( x\right)
,P_{1}\left( y\right) $ such that the polynomials $\det \left\Vert x_{i}%
\displaystyle\frac{\partial P_{0}}{\partial x_{j}}\right\Vert $ and $\det \left\Vert y_{i}%
\displaystyle\frac{\partial P_{1}}{\partial y_{j}}\right\Vert $ do not equal to zero.

Summarizing, we get the following

\begin{proposition}
Regular $%\limfunc
\frak{sp}\left( V\right) -$ orbits \ in $S^{k}\left( V^{\ast
}\right) $ has dimension $\dim %\limfunc
\frak{sp}\left( V\right) =\dim S^{2}\left(
V^{\ast }\right) ,$ when $k\geq 3.$
\end{proposition}

\begin{remark}
Codimension of regular $%\limfunc
\frak{sp}\left( V\right) -$ orbits \ in $%
S^{2}\left( V^{\ast }\right) $ equals to the dimension of the Cartan
algebra, i.e. $n=\displaystyle\frac{\dim V}{2}$. The $%\limfunc
\frak{sp}\left( V\right) -$
invariants are traces of the even powers of the operators $X_{P}.$
Stabilizers of regular $%\limfunc
\frak{sp}\left( V\right)$ - orbits in $S^{k}\left( V^{\ast }\right) ,k\geq 3,$ are discrete algebraic groups, and,
therefore, they are finite.
\end{remark}

\begin{example}
Let $\dim V=4$ and $k=3.$ Then regular $%\limfunc
\frak{sp}\left( V\right) -$
orbits \ in $S^{3}\left( V^{\ast }\right) $ has codimension $10.$ Let $P\in
S^{3}\left( V^{\ast }\right) ,$ then $P_{2}=[P,P]_{2}\in S^{2}\left( V^{\ast
}\right) $ and $I_{4}(P)=%\limfunc
{\rm Tr}\left( X_{P_{2}}^{2}\right)
,I_{8}\left( P\right) =%\limfunc
{\rm Tr}\left( X_{P_{2}}^{4}\right) ,$ are $%
%\limfunc
\frak{sp}\left( V\right) -$ invariants. Moreover, we have operators $%
A_{P}:S^{q}\left( V^{\ast }\right) \rightarrow S^{q}\left( V^{\ast }\right)
, $ where $A_{P}\left( Q\right) =[P_{2},Q]_{1},$ and traces of their powers $%
%\limfunc
{\rm Tr}\left( A_{P}^{2l}\right) $ give us $%\limfunc
\frak{sp}\left( V\right)
- $ invariants. In particular, for the case $k=3,$ we get 10 $%\limfunc
\frak{sp}%
\left( V\right) -$ invariants$: %\limfunc
{\rm Tr}\left( A_{P}^{2l}\right) ,$ $%
l=1,..,10.$
\end{example}

%======================================================
\subsection{Metric transvectants and metric invariants}

The above approach could be applied word by word to description of $%\limfunc
\frak{so}-$ invariant polynomials on an Eucledean vector space $\left( V,g\right)
, $ where $g\in S^{2}\left( V\right) $ is a metric on the dual space.

Let $\left\{ e_{1},..,e_{n}\right\} \subset V,\dim V=n,$ be an orthonormal
basis for the structure form $g\in S^{2}\left( V\right) ,$ i.e.%
\begin{equation*}
g=%\dsum
\sum\limits_{i=1}^{n}e_{i}\otimes e_{i}.
\end{equation*}%
Then the bi-differential operator $\widehat{g}:\mathcal{S\otimes
S\rightarrow S\otimes S}$ act as follows%
\begin{equation*}
\widehat{g}\left( P\otimes Q\right) =%\dsum
\sum\limits_{i=1}^{n}\left(
e_{i}\left( P\right) \otimes e_{i}\left( Q\right) \right) ,
\end{equation*}%
and the metric transvectants are 
\begin{equation*}
\left( P,Q\right) _{m}=\mu \left( \widehat{g}^{m}\left( P\otimes Q\right)
\right) ,
\end{equation*}%
and the have the following expression in the orthonormal coordinates:%
\begin{equation*}
\left( P,Q\right) _{m}=%\dsum
\sum\limits_{m_{1}+\cdots +m_{n}=m}\binom{m}{%
m_{1}...m_{n}}\frac{\partial ^{m}P}{\partial x_{1}^{m_{1}}\cdots \partial
x_{n}^{m_{m}}}\frac{\partial ^{m}Q}{\partial x_{1}^{m_{1}}\cdots \partial
x_{n}^{m_{m}}}.
\end{equation*}%
Thus,%
\begin{equation*}
\left( P,Q\right) _{1}=%\dsum
\sum\limits_{i=1}^{n}\frac{\partial P}{\partial x_{i}%
}\frac{\partial Q}{\partial x_{i}},
\end{equation*}%
and 
\begin{equation*}
\left( P,Q\right) _{2}=%\dsum
\sum\limits_{i=1}^{n}\frac{\partial ^{2}P}{\partial
x_{i}^{2}}\frac{\partial ^{2}Q}{\partial x_{i}^{2}}+2%\dsum
\sum\limits_{i\neq j}%
\frac{\partial ^{2}P}{\partial x_{i}\partial x_{j}}\frac{\partial ^{2}Q}{%
\partial x_{i}\partial x_{j}}.
\end{equation*}

Remark, that, as above, the metric transvectants generate the linear
operators 
\begin{eqnarray*}
Q &\in &S^{q}V^{\ast }\rightarrow (P,Q)_{m}\in S^{p+q-2m}V^{\ast }, \\
m &\leq &\min \left( p,q\right) ,
\end{eqnarray*}

and we get linear operators 
\begin{eqnarray*}
B_{P,q} &:&S^{q}V^{\ast }\rightarrow S^{q}V^{\ast } \\
B_{P,q} &:&Q\rightarrow (P,Q)_{m}, \\
p &=&2m,
\end{eqnarray*}%
\newline
in the case when $p=\deg \left( P\right) $ is even and $p=2m,$ $2q\geq p.$

For general degree $p,$ we substitute tensor $P\in S^{p}V^{\ast }$ by the
transvectants $P_{l}=(P,P)_{l},$ where $l<p,$ and get operators 
\begin{eqnarray*}
B_{P,l} &:&S^{q}V^{\ast }\rightarrow S^{q}V^{\ast }, \\
B_{P,l} &:&Q\rightarrow (P_{l},Q)_{p-l},
\end{eqnarray*}%
where $q\geq p-l.$

In the case, $l=p$ we get also $%\limfunc
{\rm so}\left( g\right)$ - invariant
binary forms\\ $K_p(P, Q) = (P,Q)_p$ on $S^{p}V^{\ast }$.

\begin{theorem}
Functions $M_{l,k}:S^{p}V^{\ast }\rightarrow \mathbb{R}$, 
\begin{eqnarray*}
M_{l,k}\left( P\right) &=&%\limfunc
{\rm Tr}\left( B_{P,l}^{k}\right) , \\
p-q &\leq &l\leq p,\ 1\leq k\leq \binom{n+q-1}{q},
\end{eqnarray*}%
for general degree $p,$ quadratic polynomials $K_{p},$ and functions $%
N_{k,q}:S^{p}V^{\ast }\rightarrow \mathbb{R}$ 
\begin{eqnarray*}
N_{k,q}\left( P\right) &=&%\limfunc
{\rm Tr}(B_{P,q}^{k}), \\
2q &\geq &p,\ 1\leq k\leq \binom{n+q-1}{q},
\end{eqnarray*}%
for even degree $p,$ are $%\limfunc
{\rm so}\left( V\right) -$ invariant
polynomials on $S^{p}V^{\ast }$ of degree 2k. \newline
\newline
Invariant nondegenerate symmetric forms $K_{p}\left( P,Q\right) =\left(
P,Q\right) _{p}$ on $S^{p}V^{\ast }$ realize representations $%\limfunc
{\rm so}%
\left( V\right) \rightarrow %\limfunc
{\rm so}\left( S^{p}V^{\ast }\right) .$
\end{theorem}

%%%%%%%%%%%%%%%%%%%%%%%%%%%%%%%%%%%%%%%%%%
\section{Constant type differential operators}

%=======================================================
\subsection{Connections and quantizations}

At first, we shortly remind the quantization procedure, that we have used in
(\cite{LY1}).

Let $\nabla $ be a connection in the cotangent bundle $\tau ^{\ast }:T^{\ast
}M\rightarrow M$ and let $\nabla ^{\otimes k}$ be extension of this
connection on the bundles of symmetric tensor products $\tau _{k}^{\ast
}:S^{k}T^{\ast }M\rightarrow M.$

Denote by $d_{\nabla }:\Sigma ^{k}\left( M\right) \rightarrow \Sigma
^{k}\left( M\right) \otimes \Sigma ^{1}\left( M\right) $ the covariant
differential of the last connection.

Taking symmetrization of these covariant differentials, we get derivations $%
d_{\nabla }^{s}$ into the symmetric algebra $\Sigma ^{\cdot }\left( M\right)
=\oplus _{k\geq 0}\Sigma ^{k}\left( M\right) ,$ where 
\begin{equation*}
d_{\nabla }^{s}:\Sigma ^{k}\left( M\right) \overset{d_{\nabla }}{\rightarrow 
}\Sigma ^{k}\left( M\right) \otimes \Sigma ^{1}\left( M\right) \overset{%
\text{Sym}}{\rightarrow }\Sigma ^{k+1}\left( M\right) ,
\end{equation*}%
for $k\geq 1,$ and $d_{\nabla }^{s}=d:C^{\infty }\left( M\right) \rightarrow
\Sigma ^{1}\left( M\right) ,$ for $k=0.$

Then, the $k$-th power of $d_{\nabla }^{s}$ defines a $k$-th order operator $%
\left( d_{\nabla }^{s}\right) ^{k}:C^{\infty }\left( M\right) \rightarrow
\Sigma ^{k}\left( M\right) .$

It is easy to check, that the symbol of this operator at a differential form 
$\theta \in \Sigma ^{1}\left( M\right) $ equals to the $k$-th power $\theta
^{k}\in \Sigma ^{k}\left( M\right) .$

Let now $H\in \Sigma _{k}\left( M\right) $ be a symmetric contravariant
tensor.

Denote by $Q\left( H\right) \in {\rm Diff}_{k}\left( M\right) $ the
following $k$-th order differential operator%
\begin{equation*}
Q\left( H\right) \left( f\right) =\frac{1}{k!}\left\langle H,\left(
d_{\nabla }^{s}\right) ^{k}\left( f\right) \right\rangle ,
\end{equation*}%
where $\left\langle \cdot ,\cdot \right\rangle :\Sigma _{k}\left( M\right)
\otimes \Sigma ^{k}\left( M\right) \rightarrow C^{\infty }\left( M\right) $
is the standard pairing between contra and covariant tensors.

It is easy to check, that the symbol $Q\left( H\right) $ equals $H,$ and,
therefore, the correspondence $Q:H\in \Sigma _{k}\left( M\right) \rightarrow
Q\left( H\right) \in {\rm Diff}_{k}\left( M\right) ,$ that we call 
\textit{quantization}, splits exact sequence (\ref{exact seq}).

On the other hand, let $A\in {\rm Diff}_{k}\left( M\right) $ be a
differential operator and let $\sigma _{A}\in \Sigma _{k}\left( M\right) $
be it symbol.

Let $A_{1}=A-Q\left( \sigma _{A}\right) \in {\rm Diff}_{k-1}\left(
M\right) $ and let $\sigma _{1}\in \Sigma _{k-1}\left( M\right) $ be its
symbol. Then, $A_{2}=A_{1}-Q\left( \sigma _{1}\right) \in {\rm Diff}%
_{k-2}\left( M\right) ,$ and, continue this \ way we get tensors $\sigma
_{i}\in \Sigma _{i}\left( M\right) ,$ $i=0,1,2...,k-1,$ that are calling 
\textit{subsymbols}, and representation of the initial operator $A$ in the
form%
\begin{equation}
A=Q\left( \sigma _{A}\right) +\Sigma _{i=0}^{k-1}Q\left( \sigma _{i}\right) ,
\label{splitting}
\end{equation}%
where 
\begin{multline*}%{equation*}
\sigma \left( A\right) =[\sigma _{A},\sigma _{k-1},...,\sigma _{1},\sigma
_{0}]\in \Sigma _{k}\left( M\right) \oplus \Sigma _{k-1}\left( M\right)
\oplus \cdots\\ \oplus \Sigma _{1}\left( M\right) \oplus \Sigma _{0}\left(
M\right)
\end{multline*}%{equation*}%
is called \textit{total symbol }of the operator.

\subsection{Constant type operators and associated connections}

Let now $\Gamma $ be a transitive Lie pseudogroup on manifold $M$ and let $%
A\in {\rm Diff}_{k}\left( M\right) $ be a differential operator on $M.$

Denote by $O_{a}\subset S^{k}T_{a}M$ \ the $G_{1}\left( a\right) $-orbit of
the symbol $\sigma _{A,a}\in S^{k}T_{a}M,$ where $G_{1}\left( a\right)
\subset %\limfunc
{\rm End}\left( T_{a}\right) $ , is the \textit{linear isotropy
group } (\ref{Tower}).

Denote by $[\psi ]_{a,b},$ $a,b\in M,$ the $1$-jet of a diffeomorphism $\psi
\in \Gamma ,$ such that $\psi \left( a\right) =b.$

Remark also, that $[\psi ]_{a,b}\circ \lbrack \widetilde{\psi }]_{b,a}\in
G_{1}\left( b\right) ,$ for any $\widetilde{\psi }\in \Gamma ,$ such that $%
\widetilde{\psi }\left( b\right) =a.$

We say that an operator $A\in {\rm Diff}_{k}\left( M\right) $ has the
same \textit{type} at points $a,b\in M,$ if for diffeomorphisms $\psi \in
\Gamma ,$ $\psi \left( a\right) =b,$ we have 
\begin{equation}
\lbrack \psi ]_{a,b}\left( \sigma _{A,a}\right) \in O_{b}.  \label{type}
\end{equation}%
We also say that an operator $A\in {\rm Diff}_{k}\left( M\right) $ has a 
\textit{constant }$\Gamma $\textit{-type}, if 
\begin{equation*}
\lbrack \psi ]_{a,b}\left( O_{a}\right) =O_{b},
\end{equation*}%
for any points $a,b\in M$ , and all diffeomorphisms $\psi \in \Gamma ,$ such
that $\psi \left( a\right) =b.$

An affine connection $\nabla $ on a $\Gamma $-manifold $M$ is said to be a $%
\Gamma $\textit{-connec\-tion} if the $\nabla -$parallel transports along
paths, connecting points $a,b\in M,$ are elements of $G_{1}\left( a,b\right)
.$

\begin{theorem}
Let $\Gamma $ be a transitive algebraic Lie pseudogroup on manifold $M$ and
let $A\in {\rm Diff}_{k}\left( M\right) $ be a differential operator on $%
M$ of constant $\Gamma $-type, such that the stabilizers of $\ $the linear
isotropy groups %orbits 
of the symbol $\sigma _{A}$ are finite (or equally, $\dim O_{a}=\dim G_{1}(a)$ for all 
$a\in M$). 

Then there exists
and unique an affine $\Gamma $-connection on $M,$ that preserves the symbol $%
\sigma _{A}\in \Sigma _{k}\left( M\right) $ of the operator.\newline
\end{theorem}

\begin{proof}
Let $U\ni a$ be a neighborhood, that we assume to be reasonably small, $%
G_{1}\left( U,U\right) =\delta _{1,0}^{-1}\left( U\times U\right) \subset
G_{1},$and let $f:U\times U\rightarrow $ $G_{1}\left( U,U\right) $ be a
section of the bundle $\delta _{1,0}$ over $U\times U.$ Then elements $%
f\left( b,a\right) \in G_{1}(b,a)$\ satisfy condition (\ref{type}), and $%
f\left( a,a\right) \in G_{1}\left( a\right) .$ Therefore, elements $%
\widetilde{f}\left( b,a\right) =f(a,a)^{-1}\cdot f\left( b,a\right) $
satisfy condition (\ref{type}) and, in addition, $\widetilde{f}\left(
a,a\right) $ is the unit element of the group $G_{1}\left( a\right) .$%
\newline
Let $St_{a}\subset G_{1}\left( a\right) $ be the stabilizer of the symbol $%
\sigma _{A,a}$ and $\pi :G_{1}\left( a\right) \rightarrow O_{a}=G_{1}\left(
a\right) /St_{a}$ be the natural covering. Take such a neighborhood $%
V_{a}\subset O_{a}$, that $\pi ^{-1}\left( V_{a}\right) \cap St_{a}$
contains only the unit element. Assume now, that neighborhood $U$ is so
small, that $\widetilde{f}\left( b,a\right) \left( \sigma _{A,b}\right) \in
V_{a},$ for all $b\in U.$\newline
Then, by the construction, we have such a unique map $\lambda :U\rightarrow
G_{1}\left( a\right) ,$ where $\lambda \left( a\right) $ is the unit of $%
G_{1}\left( a\right) $, that $\widetilde{f}\left( b,a\right) \left( \sigma
_{A,b}\right) =\lambda \left( a\right)\sigma _{A,a}$.\newline
Thus, the family of isomorphisms $T_{b,a}=\lambda \left( a\right) ^{-1}%
\widetilde{f}\left( a,b\right) \in G_{1}(b,a)$ is uniquely determined,
preserves the symbols: $T_{b,a}\left( \sigma _{A,b}\right) =\sigma _{A,a},$
and, therefore, determines the required affine connection (we call it 
\textit{Wagner }$\Gamma $-\textit{connection,(cf. \cite{Wag},\cite{LY2})}).
\end{proof}

Let $A\in {\rm Diff}_{k}\left( M\right) $ be, as in the above theorem, a
differential operator on $M$ of constant $\Gamma $-type and let $\nabla $ be
the Wagner connection, associated with this operator.

Let $Q:\Sigma _{i}\left( M\right) \rightarrow {\rm Diff}_{i}\left(
M\right) $ be the quantization, associated with the Wagner connection.

Denote by 
\begin{multline*}%{equation*}
\sigma \left( A\right) =[\sigma _{A},\sigma _{k-1},...,\sigma _{1},\sigma
_{0}]\in \Sigma _{k}\left( M\right) \oplus \Sigma _{k-1}\left( M\right)
\oplus \cdots\\ \oplus \Sigma _{1}\left( M\right) \oplus \Sigma _{0}\left(
M\right)
\end{multline*}%{equation*}%
the total symbol of the operator.

Remark, that $\Gamma $-equivalent operators have $\Gamma $-equivalent total
symbols.

\begin{theorem}
Let $A_{1},A_{2}\in {\rm Diff}_{k}\left( M\right) $ be, as in the above
theorem, differential operators on $M$ of constant $\Gamma $-type and let $%
\nabla _{1},\nabla _{2}$ be the Wagner connections, associated with these
operators. Then operators $A_{1},A_{2}$ are $\Gamma $-equivalent if and only
if their total symbols $\sigma _{A_{1}},\sigma _{A_{2}}\in \Sigma _{k}\left(
M\right) $ are $\Gamma $-equivalent.\newline
\end{theorem}

\begin{proof}
It is enough to note that $\Gamma $-equivalence of total symbols implies the 
$\Gamma $-equivalence of the principal symbols and, therefore, $\Gamma $%
-equivalence of the Wagner connections and spltting (\ref{splitting}).
\end{proof}

\begin{remark}
The pseudogroup of local symplectomorphisms, $\mathfrak{S}$, satisfies the
requirements of the above theorems, and, therefore, any regular linear
differential operator $A\in {\rm Diff}_{k}\left( M\right) $ of the
constant $\mathfrak{S}$-type, having order $k\geq 3,$ defines a symplectic
connection on the symplectic manifold that preserves the symbol $\sigma
_{A}\in \Sigma _{k}\left( M\right) .$
\end{remark}

\subsection{The second order operators}

In the case of the second order linear differential operators, we have in
hands two practical tools. At first, to get invariants in the case of
regular symbols, one can use the Levi-Civita connection, that is naturally
associated with the operators, and the correspondent quantization.

On the other hand, at least for some pseudogroups, the symbols themselves
have algebraic invariants with respect to the linear isotropy group, as well
as differential invariants (\cite{LY4}).

Altogether, this allows us to get the necessary number of applicable for any
pseudogroup differential invariants and apply the $n$-invariant principle.

\subsection{Weakly nonlinear operators}

In papers (\cite{LY},\cite{LY5}) we have stu\-died some class of nonlinear
operators, that we call \textit{weakly nonlinear}.

In local coordinates $\left( x_{1},..,x_{n}\right) $ these operators have
the following form: $A_{w}\left( f\right) =%\dsum
\sum\limits_{\left\vert \alpha
\right\vert \leq k}a_{\alpha }\left( x,f(x)\right) \partial ^{\alpha },$
where coefficients $a_{\alpha }\left( x,u\right) ,$ as functions on the
space of zero order jets $\mathbf{J}^{0}\left( M\right) $ belong (at any
point $x\in M$ ) to a fixed finite algebraic extension of the field $\mathbf{%
Q}\left( u\right) $ of rational in $u$ functions.

We have shown how to get natural differential invariants for such classes of
operators from invariants of linear differential operators on $\mathbf{J}%
^{0}\left( M\right) $ of the form $A=%\dsum
\sum\limits_{\left\vert \alpha
\right\vert \leq k}a_{\alpha }\left( x,u\right) \partial ^{\alpha }.$

In the case of pseudogroup $\Gamma ,$ different from the pseudogroup of all
local diffeomorphisms of $M$ , we have an additional option to use algebraic
invariants of the symbols $\sigma _{A}$ with respect to the linear isotropy
group and then apply the $n$-invariant principle.

%%%%%%%%%%%%%%%%%%%%%%%%%%%%%%%%%%%%%%%%%%

\end{document}